\documentclass[11pt,reqno]{amsart}
\usepackage{amssymb,amsmath,graphicx,amsfonts,euscript}
\usepackage{amsmath}

\usepackage[dvipsnames,usenames]{xcolor}
\newtheorem{thm}{Theorem}[section]

\newtheorem{prop}[thm]{Proposition}

\newtheorem{rem}[thm]{Remark}

\numberwithin{equation}{section}

\date{\today}

\def\Om{\Omega}

\subjclass[2020]{35Q35, 35B35, 76D05, 76D10}

\begin{document}
\title[Vanishing vertical viscosity in 2D NSE]
{Vanishing Vertical Viscosity in Two-Dimensional Anisotropic Navier-Stokes Equations with No-Slip Boundary Conditions: An $L^p$ result}

\date{July 2, 2025}

\author[C. Cao]{Chongsheng Cao}
\address{Department of Mathematics \& Statistics \\Florida International University\\
Miami, Florida 33199, USA} \email{caoc@fiu.edu}

\author[Y. Guo]{Yanqiu Guo}
\address{Department of Mathematics \& Statistics \\Florida International University\\
Miami, Florida 33199, USA} \email{yanguo@fiu.edu}

\begin{abstract}
This paper studies the inviscid limit problem for the two-dimensional Navier-Stokes equations with anisotropic viscosity. The fluid is assumed to be bounded above and below by impenetrable walls, with a no-slip boundary condition imposed on the bottom wall. For $H^2$ initial velocity, we establish strong convergence in the $L^p$ norm to the limiting problem as the vertical viscosity approaches zero, for any $2\leq p <\infty$. The main challenge lies in the mismatch of boundary conditions—specifically, the no-slip condition in the original problem versus the slip condition in the limiting problem.
\end{abstract}

\maketitle
\section{Introduction}\label{intro}

\subsection{The model}
We consider the two-dimensional \emph{anisotropic} Navier-Stokes equations
\begin{equation}\label{vbou}
 \left\{
\begin{array}{l}
 u_t+uu_x+vu_y+ p_x +q_x = \nu_1 u_{xx} + \nu_2 u_{yy}, \\
 v_t+uv_x+vv_y+ p_y +q_y =  \nu_1 v_{xx} + \nu_2 v_{yy}, \\
 u_x+v_y=0, \\
 u(x,0,t)=u_y(x,1, t)= v(x,0,t)=v(x,1,t)=0,\\
 u(x,y,0)=u_0(x,y), \quad v(x,y,0)=v_0(x,y),
 \end{array} \right.
 \end{equation}
where $u, v, p+q$ are scalar functions of $(x,y)\in \Om = [-\pi, \pi]\times [0,1]$ and $t\ge 0$. These functions are periodic in the $x$ variable.
Here, $\nu_1$ and $\nu_2$ represent horizontal viscosity and vertical viscosity, respectively. In this article, we assume that $\nu_2 \ll  \nu_1$.
In addition, $(u,v)$ denotes the 2D velocity field, and $p+q$ represents the pressure. More precisely, the pressure is decomposed into two parts:
the flow pressure $p$ that satisfies (\ref{PP}), and the boundary layer pressure $q$ that satisfies (\ref{PQ}).
At the boundaries $y=0$ and $y=1$, the vertical velocity component $v$ is zero. Physically, this ensures that there is no fluid flow normal to the top and bottom walls.
At the bottom boundary $y=0$, the horizontal velocity component $u=0$. This represents a \emph{no-slip} condition at the bottom wall. At the top of the boundary $y=1$, the vertical derivative of the horizontal velocity vanishes, i.e., $u_y=0$.
Physically, this represents a free-slip condition at the top wall, where there is no shear stress acting on the fluid.

When $\nu_2=0$, we obtain the 2D NSE with only horizontal viscosity:
\begin{equation}\label{hbou}
 \left\{
\begin{array}{l}
 U_t+UU_x+VU_y+ P_x = \nu_1 U_{xx}, \\
 V_t+UV_x+VV_y+ P_y=  \nu_1 V_{xx}, \\
 U_x+V_y=0, \\
 V(x,0,t)=V(x,1,t)=0,\\
 U(x,y,0)=U_0(x,y), \quad V(x,y,0)=V_0(x,y).
 \end{array} \right.
 \end{equation}
In a recent paper \cite{Cao}, we have shown the global well-posedness of weak solutions to (\ref{hbou}) for initial data $U_0, V_0\in L^2(\Omega)$ with $\partial_y U_0 \in L^2(\Omega)$, as well as a uniform bound on the $H^2$-norm of strong solutions with $H^2$ initial data. 
In the present manuscript, we investigate the convergence of the solution of system (\ref{vbou}) to the solution of system (\ref{hbou}) as $\nu_2 \rightarrow 0^+$ for sufficiently regular initial data. Due to the mismatch in boundary conditions, one cannot expect convergence in a norm stronger than the $L^p-$norm, $2\leq p\leq \infty$.
The main result of this article is on the convergence of the solution of system (\ref{vbou}) to the solution of system (\ref{hbou}) as $\nu_2 \rightarrow 0^+$ in the $L^p$ norm for $2\leq p<\infty$, by assuming that the initial velocity is in $H^2(\Omega)$.

The study of vanishing viscosity in the NSE is a central topic in fluid dynamics. It involves analyzing the behavior of solutions to the NSE as the viscosity tends to zero, which connects the NSE to the Euler equations. Mathematically, the inviscid limit is a singular perturbation problem, as the highest-order term is formally dropped from the equations in the limit. The convergence is highly dependent on the boundary conditions and the regularity of the solutions. In the absence of boundaries, convergence can be established for smooth initial data \cite{ Ebin, Gol,  Kato-72, Swann} and for singular solutions \cite{Chemin-96, C-Wu}. However, in the presence of boundaries, the mismatch between the no-slip boundary conditions of the NSE and the slip boundary conditions of the Euler equations complicates the analysis. In this case, a boundary layer forms near the boundary, where the effects of viscosity cannot be neglected even at very low viscosity. The concept of a viscous boundary layer was introduced by Prandtl \cite{Prandtl} in 1904 under the no-slip boundary condition. The basic idea is that the fluid region can be divided into two parts: a thin layer close to the boundary, where the effect of viscosity is significant, and the outer region, where viscosity can be neglected. The thickness of the boundary layer is formally estimated as $O(\sqrt{\nu})$. The fundamental equations describing the boundary layer are the Prandtl equations. A classical result by Sammartino and Caflisch \cite{Caf} states that the Prandtl asymptotic expansion is valid for a short time for analytic data. 
Maekawa \cite{Mae} showed that the zero-viscosity limit holds for a short time when using the vorticity formulation in the half-plane, provided the initial vorticity remains bounded away from the boundary. This result has been extended to the half-space in $\mathbb R^3$ by Fei et al. \cite{Fei}. Mathematically, the main difficulty in the case of the no-slip boundary condition is the lack of a priori estimates in strong enough norms to pass to the limit. The classical Kato’s criterion \cite{Kato-84} states that the vanishing of energy dissipation in a small layer near the boundary is equivalent to the validity of the zero-viscosity limit in the energy space. Whether the vanishing viscosity limit holds generically, even for a short time under no-slip boundary conditions, remains largely an open problem.

Regarding the anisotropic Navier--Stokes equations in 3D, Masmoudi~\cite{Mas} studied the inviscid limit problem for weak solutions with distinct horizontal and vertical viscosities as they vanish at different rates. Specifically, he considered the case where the ratio of vertical to horizontal viscosity tends to zero. Chemin et al.~\cite{CDGG} investigated the vanishing vertical viscosity limit for the anisotropic NSE in the absence of a physical boundary. Iftimie and Planas~\cite{Iftimie} examined the half-space 3D flow under a Navier friction condition, where only the vertical viscosity vanishes, and proved convergence to an appropriate limit flow. Liu and Wang~\cite{Liu} analyzed the asymptotic behavior of solutions to the anisotropic NSE with a no-slip boundary condition in a half-space of $\mathbb{R}^3$ as the vertical viscosity tends to zero. Tao~\cite{Tao} studied the vanishing vertical viscosity limit for the anisotropic NSE with the no-slip boundary condition in the half-plane, rigorously justifying the asymptotic expansion in both the energy space and the $L^{\infty}$ space while deriving the optimal convergence rate.

Concerning the inviscid limit problem in the $L^p$ setting, Beir\~{a}o da Veiga and Crispo~\cite{BC} established the strong convergence of solutions for the 3D NSE to those of the Euler equations in a periodic channel with Navier boundary conditions, specifically in the space $W^{s,p}$ for $s<3$ and $p>3/2$. More recently, Aydin~\cite{Aydin} derived uniform bounds and proved the inviscid limit in $L^p$-based conormal Sobolev spaces for solutions of the NSE with Navier boundary conditions in the half-space. We emphasize that the no-slip boundary condition assumed in the present work differs fundamentally from the Navier boundary condition. Finally, we mention that Bardos et al. \cite{BNNT} proved the inviscid limit under the $L^p$ norm for 2D NSE on a general bounded domain with no-slip boundary conditions for initial data that are analytic near the boundary.

\vspace{0.1 in}

\subsection{The main result.} Our main result states that the solution of system (\ref{vbou}) converges strongly to the solution of system (\ref{hbou}) in the $L^p(\Omega)$ norm for any $2\leq p<\infty$, as the vertical viscosity $\nu_2 \rightarrow 0$, provided the initial data are sufficiently regular.

Throughout the paper, we denote $\|u\|_p = \|u\|_{L^p(\Omega)}$, where $\Omega = [-\pi,\pi] \times [0,1]$.

\begin{thm} \label{main-thm}
Let $(u,v)$ be the unique solution of system (\ref{vbou}) with initial data $u_0,v_0 \in H^2(\Omega)$. Let $(U,V)$ be the unique solution of system (\ref{hbou}) with initial data $U_0,V_0 \in H^2(\Omega)$.
Then, for any $T>0$, the following estimates hold:

\begin{enumerate}
\item  \begin{align}
&\int_0^T \|\nabla p\|_2^2 \, dt \leq M(\nu_1), \label{p-ebase} \\
&\int_0^T \|\nabla q\|_2^2 \, dt \leq \nu_2^{1/2} M(\nu_1). \label{q-ebase2}
\end{align}

\item For any $r\geq 2$,
\begin{align}
&\max_{0\leq t \leq T} (\|u\|_r + \|v\|_r)  \leq M(\nu_1, r, T), \label{u-ebase} \\
&\max_{0\leq t \leq T} \|v_x\|_2  \leq M(\nu_1). \label{v-x-ebase2}
\end{align}

\item  \begin{align}
\max_{0\leq t \leq T} \|u\|_{\infty}^2 &\leq M(\nu_1, T)  (e+ |\log \nu_2| ) \log (e+ |\log \nu_2| ). \label{u-infty-ebase}
\end{align}
\end{enumerate}
Furthermore,
\begin{align} \label{L2-dif}
\sup_{t\in [0,T]} (\| u - U\|_2 + \|v-V\|_2) \leq e^{M(\nu_1,  T)}\left(\|u_0 - U_0\|_2 +  \|v_0-V_0\|_2  +  \nu_2^{1/8}  \right).
\end{align}
In addition, if $(u_0,v_0)=(U_0,V_0)$, then
\begin{align} \label{Lp-dif}
\sup_{t\in [0,T]}(\|u-U\|_r + \|v-V\|_r) \rightarrow 0, \quad \text{as} \quad  \nu_2 \rightarrow 0,
\end{align}
for any $2\leq r <\infty$.
\end{thm}

\vspace{0.1 in}

\begin{rem}
Our inviscid limit result stated in Theorem \ref{main-thm} has three notable features:
\begin{enumerate}
\item Equation (\ref{vbou}) describes an anisotropic 2D Navier–Stokes system, and we consider the limit as only the vertical viscosity approaches zero.
\item We impose the no-slip boundary condition at the bottom wall of the channel, which gives rise to a boundary layer near that wall.
\item The convergence holds in the $L^p$ norm for all $2 \leq p < \infty$.
\end{enumerate}
\end{rem}

\vspace{0.1 in}

\begin{rem}
The global well-posedness of weak solutions to equation (\ref{hbou}) with initial data $U_0, V_0 \in L^2(\Omega)$ and $\partial_y U_0 \in L^2(\Omega)$ was established in our recent paper \cite{Cao}. Furthermore, in \cite{Cao}, we also showed that for $U_0, V_0 \in H^2(\Omega)$, equation (\ref{hbou}) admits a unique global strong solution with a uniform bound $\|U\|_{H^2} + \|V\|_{H^2} \leq C$ for all $t \geq 0$, where the bound $C$ depends on the initial data but does not depend on time.

\end{rem}

\vspace{0.2 in}

\section{Estimates} \label{sec-est}

First, we present the following estimate for the triple product in terms of the Lebesgue norms of the functions and their derivatives in different directions. This inequality will be used frequently throughout the paper. Such inequalities are crucial for analyzing anisotropic NSE.
\begin{prop} \label{triple}
Let $m\geq 1$. Assume that $f, g, g_y, h_x\in L^2(\Om)$ and $h\in L^{2m}(\Om)$. Then
$$
 \int_{\Om}| f \, g\, h|  \;dxdy  \le C \, \|f\|_2 \, \|g\|_{2}^{\frac{m}{m+1}} (\|g\|_{2} +\|g_y\|_2)^{\frac{1}{m+1}} \,\|h\|_{2m}^{\frac{m}{m+1}} (\|h\|_{2}+\|h_x\|_2)^{\frac{1}{m+1}},
$$
where $C$ is a constant independent of $f$, $g$ and $h$.
\end{prop}

\begin{proof}
The following formal calculation can be justified using a density argument. By applying Hölder's inequality, the Gagliardo–Nirenberg interpolation inequality, and Minkowski's integral inequality, we deduce:
\begin{eqnarray*}
&&
 \int_{\Om}| f \, g\, h|  \;dxdy  \le  \, \int_0^1 \|f\|_{L^2(-\pi,\pi)}\, \|g\|_{L^2(-\pi,\pi)} \,\|h\|_{L^{\infty}(-\pi,\pi)} \; dy \\
 && \le C \int_0^1 \,  \|f\|_{L^2(-\pi,\pi)}\, \|g\|_{L^2(-\pi,\pi)}  \,\|h\|_{L^{2m}(-\pi,\pi)}^{\frac{m}{m+1}}  \left(\|h\|_{L^2(-\pi,\pi)}+\|h_x\|_{L^2(-\pi,\pi)}\right)^{\frac{1}{m+1}} \; dy \\
 && \le C  \, \|f\|_2 \,  \,\|h\|_{2m}^{\frac{m}{m+1}}  (\|h\|_{2}+\|h_x\|_{2})^{\frac{1}{m+1}} \;
 \left\| \|g\|_{L^2(-\pi,\pi)} \right\|_{L^{\frac{2(m+1)}{m-1}}(0,1)}  \\
 && \le C  \, \|f\|_2 \,  \,\|h\|_{2m}^{\frac{m}{m+1}}  (\|h\|_{2}+\|h_x\|_{2})^{\frac{1}{m+1}} \;
\Big \| \|g \|_{L^{\frac{2(m+1)}{m-1}} (0,1)} \Big\|_{L^2(-\pi,\pi)}  \\
 && \le C  \, \|f\|_2 \,  \,\|h\|_{2m}^{\frac{m}{m+1}}  (\|h\|_{2}+\|h_x\|_{2})^{\frac{1}{m+1}} \;
 \|g \|_2^{\frac{m}{m+1}}  (\|g\|_{2}+\|g_y\|_{2})^{\frac{1}{m+1}}.
\end{eqnarray*}
\end{proof}

\vspace{0.1 in}

In what follows, we derive energy estimates for solutions of the 2D anisotropic NSE (\ref{vbou}). Since we assume $H^2$ initial data in Theorem \ref{main-thm}, the solutions are sufficiently regular to justify all the estimates below. The goal of these estimates is to establish uniform bounds for the $L^p$ norms that are \emph{independent} of $\nu_2$. This independence is crucial, as in the next section we prove the vanishing vertical viscosity limit by letting $\nu_2$ approach zero.

\subsection{$\| u\|_2$ and $\| v\|_2$ Estimates}

Taking the inner product of the first equation in (\ref{vbou}) with $u$, and the second equation with $v$, and integrating by parts, we obtain
\begin{eqnarray*}
\frac{1}{2}\frac{d}{dt} (\|u\|_2^2 + \|v\|_2^2) + \nu_1 (\|u_x\|_2^2 + \|v_x\|_2^2) + \nu_2 (\|u_y\|_2^2 + \|v_y\|_2^2) = 0.
\end{eqnarray*}

Integrating over the interval $[0, t]$, we get
\begin{align} \label{L2}
\|u\|_2^2 + \|v\|_2^2 + 2\nu_1 \int_0^t (\|u_x\|_2^2 + \|v_x\|_2^2) \, ds
+ 2\nu_2 \int_0^t (\|u_y\|_2^2 + \|v_y\|_2^2) \, ds = \|u_0\|_2^2 + \|v_0\|_2^2,
\end{align}
for all $t\geq 0$.

\vspace{0.1 in}

\subsection{$\|\nabla p\|_2$ and $\|\nabla q\|_2$ Estimates}

Since we study fluid flow in a channel domain with a no-slip boundary condition on the boundary $y = 0$, the pressure can be decomposed into two parts: the flow pressure $p$ and the boundary layer pressure $q$. We recover $p(x, y, t)$ and $q(x, y, t)$ by solving the following elliptic systems:
\begin{equation}\label{PP}
\left\{
\begin{array}{l}
\Delta p = -(uu_x + vu_y)_x - (uv_x + vv_y)_y = -2(u u_x)_x - 2(u v_x)_y, \\
p_y(x, 1, t) = p_y(x, 0, t) = 0, \\
\displaystyle\int_{\Omega} p(x, y, t) \, dxdy = 0,
\end{array}
\right.
\end{equation}
and
\begin{equation}\label{PQ}
\left\{
\begin{array}{l}
\Delta q = 0, \\
q_y(x, 1, t) = 0, \quad q_y(x, 0, t) = \nu_2 v_{yy}(x, 0, t), \\
\displaystyle\int_{\Omega} q(x, y, t) \, dxdy = 0.
\end{array}
\right.
\end{equation}

By standard elliptic theory, we have
\begin{align}
\|p\|_2 &\leq C \|u^2 + v^2\|_2, \label{PQ1} \\
\|\nabla p\|_2 &\leq C(\|u u_x\|_2 + \|u v_x\|_2). \label{PQ2}
\end{align}

Next, we derive a bound for $q(x, y, t)$. Since we assume that the flow is periodic in the $x$-variable, we can write
\begin{align*}
q(x, y, t) = \sum_{k \neq 0} \hat{q}^k(y,t) e^{ikx}.
\end{align*}
Then, due to (\ref{PQ}), we have
\begin{align}
&\hat{q}^k_{yy} - k^2 \hat{q}^k = 0, \label{qODE1}\\
&\hat{q}^k_y(1,t) = 0, \quad \hat{q}^k_y(0,t) = \nu_2 \hat{v}^k_{yy}(0,t) = -i \nu_2 k \hat{u}^k_y(0,t). \label{qODE2}
\end{align}

For each fixed $t\geq 0$, solving the second-order ODE (\ref{qODE1})-(\ref{qODE2}), we obtain
\begin{align}
\hat{q}^k(y,t) = \frac{i \nu_2 \hat{u}^k_y(0,t)}{1 - e^{-2k}} \left(e^{ky - 2k} + e^{-ky}\right), \label{QQ}
\end{align}
for $y\in [0,1]$ and $k \neq 0$. It follows from (\ref{QQ}) that
\begin{align} \label{QQ'}
\left| \hat{q}^k(y,t) \right| \leq C \nu_2  \left|\hat{u}^k_y(0,t) \right|.
\end{align}

Observe that
\begin{align*}
0 = -\int_{\Omega} q \Delta q \, dxdy = \int_{\Omega} |\nabla q|^2 \, dxdy + \int_{-\pi}^{\pi} q_y(0) q(0) \, dx.
\end{align*}
Therefore, using (\ref{qODE2}) and (\ref{QQ'}), we have
\begin{align}
\|\nabla q\|_2^2 & = \left| \int_{-\pi}^{\pi} q_y(0) q(0) \, dx \right| \leq C \nu_2^2 \sum_{|k| \neq 0} |k| \, |\hat{u}^k_y(0)|^2 \leq C \nu_2^2 \sum_{|k| \neq 0} \left| k \int_0^1 \frac{d}{dy} \left(|\hat{u}^k_y|^2 \right) dy \right| \notag\\
 &\leq C \nu_2^2 \sum_{|k| \neq 0} \int_0^1 | \hat{u}^k_{yy} \hat{u}^k_y k | \, dy \leq C \nu_2^2 \|u_{yy}\|_2 \|u_{yx}\|_2. \label{grad-q}
\end{align}

\vspace{0.1 in}

\subsection{$\|u_y\|_2$ Estimates}

Taking the inner product of the first equation in (\ref{vbou}) with $-u_{yy}$ and integrating by parts, we obtain
\begin{align} \label{ebase-4-0}
&\frac{1}{2}\frac{d}{dt} \|u_y\|_2^2 + \nu_1 \|u_{xy}\|_2^2 + \nu_2 \|u_{yy}\|_2^2 \notag \\
&\leq (\|p_x\|_2 + \|q_x\|_2) \|u_{yy}\|_2 \notag \\
&\leq C \left( \int_{\Omega} (u^2 + v^2)(|u_x|^2 + |v_x|^2) \, dxdy \right)^{1/2} \|u_{yy}\|_2  + C \nu_2 \|u_{yx}\|_2^{1/2} \|u_{yy}\|_2^{3/2},
\end{align}
where we used estimates (\ref{PQ2}) and (\ref{grad-q}).

Multiplying both sides of (\ref{ebase-4-0}) by $\nu_1 \nu_2$ yields
\begin{align*}
&\nu_1 \nu_2 \frac{d}{dt} \|u_y\|_2^2 + 2\nu_1^2 \nu_2 \|u_{xy}\|_2^2 + 2\nu_1 \nu_2^2 \|u_{yy}\|_2^2  \notag\\
&\leq C \nu_1 \nu_2 \left( \int_{\Omega} (u^2+v^2)(|u_x|^2 + |v_x|^2) \, dxdy \right)^{1/2} \|u_{yy}\|_2 + C \nu_1 \nu_2^2 \|u_{yx}\|_2^{1/2} \|u_{yy}\|_2^{3/2}.
\end{align*}

Applying Young's inequality, we obtain
\begin{align*}
\nu_1 \nu_2 \frac{d}{dt} \|u_y\|_2^2 + (2\nu_1 - C \nu_2) \nu_1 \nu_2 \|u_{xy}\|_2^2 + \nu_1 \nu_2^2 \|u_{yy}\|_2^2 \leq C \nu_1 \int_{\Omega} (u^2 + v^2)(|u_x|^2 + |v_x|^2) \, dxdy.
\end{align*}
Recall we assume $\nu_1 \gg \nu_2$ throughout the paper, so that $2\nu_1 - C \nu_2 \geq \nu_1$. As a result,
\begin{align} \label{ebase-4-1}
\nu_1 \nu_2 \frac{d}{dt} \|u_y\|_2^2 + \nu_1^2 \nu_2 \|u_{xy}\|_2^2 + \nu_1 \nu_2^2 \|u_{yy}\|_2^2 \leq C \nu_1 \int_{\Omega} (u^2 + v^2)(|u_x|^2 + |v_x|^2) \, dxdy.
\end{align}

\vspace{0.1 in}

\subsection{$\|u^2 + v^2\|_2$ Estimates}

Taking the inner product of the first equation in (\ref{vbou}) with $u(u^2+v^2)$, and the second equation with $v(u^2+v^2)$, and integrating by parts, we obtain
\begin{align} \label{u2-1}
& \frac{d}{dt}  \|u^2 + v^2\|_2^2
+ 4\nu_1 \int_{\Omega} (u^2 + v^2)(|u_x|^2 + |v_x|^2) \, dxdy + 4\nu_2 \int_{\Omega} (u^2 + v^2)(|u_y|^2 + |v_y|^2) \, dxdy  \notag \\
&\leq  4 \left|\int_{\Omega} p_x\, u (u^2 + v^2) \, dxdy \right|
+ 4 \left|\int_{\Omega} p_y\, v(u^2 + v^2) \, dxdy \right| \notag \\
&\quad + 4 \left| \int_{\Omega} q_x\, u(u^2 + v^2) \, dxdy  \right|
+ 4 \left| \int_{\Omega} q_y\, v(u^2 + v^2) \, dxdy \right|.
\end{align}

The terms on the right-hand side can be estimated as follows. By Proposition \ref{triple}, we have
\begin{align}   \label{u2-2}
& 4\left| \int_{\Omega} p_x\, u(u^2 + v^2) \, dxdy \right|
\leq  C\left| \int_{\Omega} p\, u (u u_x + v v_x) \, dxdy \right|
+ C \left| \int_{\Omega} p\, u_x (u^2 + v^2) \, dxdy \right| \notag \\
&\le C \|p\|_2^{1/2} \|p_y\|_2^{1/2} \|u\|_2^{1/2} \|u_x\|_2^{1/2} \|u u_x + v v_x\|_2  + C \|p\|_2^{1/2} \|p_y\|_2^{1/2} \|u_x\|_2 \|u^2 + v^2\|_2^{1/2} \|u u_x + v v_x\|_2^{1/2} \notag\\
&\leq  \frac{1}{2}\nu_1 \int_{\Omega} (u^2 + v^2)(|u_x|^2 + |v_x|^2) \, dxdy  +   C \nu_1^{-3} \|u^2 + v^2\|_2^2 (\|u\|_2^2+1)  \|u_x\|_2^2,
\end{align}
where we have used (\ref{PQ1})-(\ref{PQ2}) and Young's inequality.

Similarly,
\begin{align} \label{u2-4}
&4\left| \int_{\Omega} p_y \,  v(u^2 + v^2) \, dxdy \right|
\le C \|p_y\|_2 \|v\|_2^{1/2} \|v_y\|_2^{1/2} \|u^2 + v^2\|_2^{1/2} \|u u_x + v v_x\|_2^{1/2} \notag\\
& \leq  \frac{1}{2}\nu_1 \int_{\Omega} (u^2 + v^2)(|u_x|^2 + |v_x|^2) \, dxdy  +   C \nu_1^{-3}  \|v\|_2^2 \,   \|u_x\|_2^2 \,  \|u^2 + v^2\|_2^2.
\end{align}

Additionally, using Proposition \ref{triple} and estimate (\ref{grad-q}), we have
\begin{align} \label{u2-5}
&4 \left| \int_{\Omega} q_x \, u \, u^2 \, dxdy \right|
\le C \|\nabla q\|_2 \|u\|_2^{1/2} \|u_y\|_2^{1/2} \|u^2\|_2^{1/2} \|u u_x\|_2^{1/2} \notag \\
&\le C \nu_2 \|u_{yy}\|_2^{1/2} \|u_{yx}\|_2^{1/2} \|u\|_2^{1/2} \|u_y\|_2^{1/2} \|u^2\|_2^{1/2} \|u u_x\|_2^{1/2} \notag \\
&\le C (\nu_2 \|u_{yx}\|_2^2)^{1/4} (\nu_2^2 \|u_{yy}\|_2^2)^{1/4} \|u\|_2^{1/2} (\nu_2 \|u_y\|_2^2)^{1/4} \|u^2\|_2^{1/2} \|u u_x\|_2^{1/2} \notag\\
&\leq   \frac{1}{2}\nu_1  \|u u_x\|_2^2  +  \frac{1}{4} \epsilon \nu_1^2 \nu_2 \|u_{yx}\|_2^2 + \frac{1}{4} \epsilon \nu_1 \nu_2^2 \|u_{yy}\|_2^2
+ C_{\epsilon} \nu_1^{-4} \nu_2 \|u_y\|_2^2 \, \|u\|_2^2 \,  \|u^2\|_2^2,
\end{align}
where $\epsilon>0$ will be chosen later.

Similarly, we deduce
\begin{align} \label{u2-6}
&4\left| \int_{\Omega} q_x \, u \, v^2 \, dxdy \right| + 4\left|\int_{\Omega} q_y \, v (u^2 + v^2) \, dxdy \right|
\le C\int_{\Omega} |q_x + q_y| \, |v| \, (u^2 + v^2) \, dxdy \notag \\
&\le C \|\nabla q\|_2 \|v\|_2^{1/2} \|u_x\|_2^{1/2} \|u^2 + v^2\|_2^{1/2} \left( \int_{\Omega} (u^2 + v^2)(|u_x|^2 + |v_x|^2) \, dxdy \right)^{1/4} \notag \\
&\le C \nu_2 \|u_{yy}\|_2^{1/2} \|u_{yx}\|_2^{1/2} \|v\|_2^{1/2} \|u_x\|_2^{1/2} \|u^2 + v^2\|_2^{1/2}  \left( \int_{\Omega} (u^2 + v^2)(|u_x|^2 + |v_x|^2) \, dxdy \right)^{1/4} \notag \\
&\le C (\nu_2 \|u_{yx}\|_2^2)^{1/4} (\nu_2^2 \|u_{yy}\|_2^2)^{1/4} \|v\|_2^{1/2} (\nu_2 \|u_x\|_2^2)^{1/4} \|u^2 + v^2\|_2^{1/2} \notag \\
&\quad \times \left( \int_{\Omega} (u^2 + v^2)(|u_x|^2 + |v_x|^2) \, dxdy \right)^{1/4} \notag\\
&\leq  \frac{1}{2}\nu_1 \int_{\Omega} (u^2 + v^2)(|u_x|^2 + |v_x|^2) \, dxdy + \frac{1}{4} \epsilon \nu_1^2 \nu_2 \|u_{yx}\|_2^2 + \frac{1}{4}\epsilon \nu_1 \nu_2^2 \|u_{yy}\|_2^2    \notag\\
&\quad  + C_{\epsilon} \nu_1^{-4} \nu_2 \|u_x\|_2^2  \, \|v\|_2^2  \,  \|u^2 + v^2\|_2^2.
\end{align}

It follows from (\ref{u2-1})-(\ref{u2-6}) that
\begin{align}  \label{u2-6'}
&\frac{d}{dt} \|u^2 + v^2\|_2^2 + 2\nu_1 \int_{\Omega} (u^2 + v^2)(|u_x|^2 + |v_x|^2) \, dxdy + 4\nu_2 \int_{\Omega} (u^2 + v^2)(|u_y|^2 + |v_y|^2) \, dxdy \notag\\
&\leq  C_{\epsilon} \nu_1^{-4} \|u^2 + v^2\|_2^2 (\|u\|_2^2+  \|v\|_2^2  +  1)  (\nu_1 \|u_x\|_2^2 +   \nu_2 \|u_y\|_2^2)   \notag\\
& \quad + \frac{1}{2} \epsilon \nu_1^2 \nu_2 \|u_{yx}\|_2^2 +  \frac{1}{2} \epsilon \nu_1 \nu_2^2 \|u_{yy}\|_2^2.
\end{align}

Multiplying (\ref{ebase-4-1}) by $\epsilon$ gives
\begin{align} \label{ebase-4-2}
\epsilon \nu_1 \nu_2 \frac{d}{dt} \|u_y\|_2^2 +   \epsilon   \nu_1^2 \nu_2 \|u_{xy}\|_2^2 +   \epsilon   \nu_1 \nu_2^2 \|u_{yy}\|_2^2
\leq      C   \epsilon     \nu_1 \int_{\Omega} (u^2 + v^2)(|u_x|^2 + |v_x|^2) \, dxdy.
\end{align}

Adding (\ref{u2-6'}) and (\ref{ebase-4-2}), and selecting $\epsilon > 0$ sufficiently small, we obtain
\begin{align*}
&\frac{d}{dt} \left( \|u^2 + v^2\|_2^2 + \epsilon \nu_1 \nu_2 \|u_y\|_2^2 \right)
+ \nu_1 \int_{\Omega} (u^2 + v^2)(|u_x|^2 + |v_x|^2) \, dxdy \\
&\quad + 4\nu_2 \int_{\Omega} (u^2 + v^2)(|u_y|^2 + |v_y|^2) \, dxdy + \frac{1}{2}\epsilon  \nu_1^2 \nu_2 \|u_{xy}\|_2^2 + \frac{1}{2}\epsilon \nu_1 \nu_2^2 \|u_{yy}\|_2^2 \\
&\leq C_{\epsilon} \nu_1^{-4} \|u^2 + v^2\|_2^2 (\|u\|_2^2+  \|v\|_2^2  +  1)  (\nu_1 \|u_x\|_2^2 +   \nu_2 \|u_y\|_2^2).
\end{align*}

Thanks to Gr\"onwall's inequality, it follows that
\begin{align} \label{u2-7}
&\|u^2 + v^2\|_2^2 + \epsilon \nu_1 \nu_2 \|u_y\|_2^2
+ \nu_1 \int_0^t \int_{\Omega} (u^2 + v^2)(|u_x|^2 + |v_x|^2) \, dxdy \, ds \notag \\
&\quad + \nu_2 \int_0^t \int_{\Omega} (u^2 + v^2)(|u_y|^2 + |v_y|^2) \, dxdy \, ds + \epsilon \int_0^t \left( \nu_1^2 \nu_2 \|u_{xy}\|_2^2 + \nu_1 \nu_2^2 \|u_{yy}\|_2^2 \right) ds \notag \\
&\leq e^{C_{\epsilon} \nu_1^{-4} (\|u_0\|_2^2 + \|v_0\|_2^2 +1 ) \int_0^t (\nu_1 \|u_x\|_2^2 + \nu_2 \|u_y\|_2^2) \, ds} \left( \|u_0^2 + v_0^2\|_2^2 + \epsilon \nu_1 \nu_2 \|\partial_y u_0\|_2^2 \right) \notag \\
&\leq  e^{C_{\epsilon} \nu_1^{-4} (\|u_0\|_2^4 + \|v_0\|_2^4 +1 ) } \left( \|u_0^2 + v_0^2\|_2^2 + \epsilon \nu_1^2 \|\partial_y u_0\|_2^2 \right) =: M(\nu_1),
\end{align}
for all $t\geq 0$, where we have used estimate (\ref{L2}) and $\nu_1 \gg \nu_2$. Notice that the bound $M(\nu_1)$ is independent of $\nu_2$. In particular, inequality (\ref{u2-7}) shows that
\begin{align} \label{u2-8}
\int_0^t \nu_2^2 \|u_{yy}\|_2^2 \, ds \leq M(\nu_1).
\end{align}
We emphasize that the uniform bound (\ref{u2-8}) is important for proving that the solution of system (\ref{vbou}) converges to the solution of system (\ref{hbou}) in the $L^p$ norm as $\nu_2 \to 0$ in Section \ref{sec-vanish}.

Moreover, due to (\ref{PQ2}), (\ref{grad-q}), and (\ref{u2-7}), we have the estimates
\begin{align}
\int_0^t \|\nabla p\|_2^2 \, ds &\leq M(\nu_1), \label{ebase} \\
\int_0^t \|\nabla q\|_2^2 \, ds &\leq \nu_2^{1/2} M(\nu_1). \label{ebase2}
\end{align}

\vspace{0.1 in}

\subsection{$\| v_x\|_2$ Estimates}

Taking the inner product of the second equation in (\ref{vbou}) with $-v_{xx}$ and integrating by parts, we obtain
\begin{align*}
\frac{1}{2}\frac{d}{dt} \|v_x\|_2^2 + \nu_1 \|v_{xx} \|_2^2 + \nu_2 \|v_{xy} \|_2^{2}
&\leq  (\|p_y\|_2 + \|q_y\|_2) \|v_{xx}\|_2 \notag\\
&\leq \frac{1}{2} \nu_1  \|v_{xx} \|_2^2 + C \nu_1^{-1}  (\|p_y\|_2^2 + \|q_y\|_2^2).
\end{align*}
It follows that
\begin{eqnarray*}
\frac{d}{dt} \|v_x\|_2^2 + \nu_1 \|v_{xx}\|_2^2 + \nu_2 \|v_{xy}\|_2^{2}
\leq C   \nu_1^{-1}  (\|\nabla p\|_2^2 + \|\nabla q\|_2^2).
\end{eqnarray*}

Integrating over time and using estimates (\ref{ebase})--(\ref{ebase2}) gives
\begin{eqnarray*}
\|v_x\|_2^2 + \nu_1 \int_0^t \|v_{xx}\|_2^2 \, ds + \nu_2 \int_0^t \|v_{xy}\|_2^{2} \, ds
\leq \|\partial_x v_0\|_2^2 + M(\nu_1),
\end{eqnarray*}
for all $t\geq 0$.

\vspace{0.1 in}

\subsection{$\| p\|_3$ Estimates}

Taking the inner product of the first equation in (\ref{vbou}) with $u^5$ and the second equation with $v^5$, and integrating by parts, we obtain:
\begin{align*}
&  \frac{1}{6}\frac{d}{dt} \int_{\Omega} (u^6 + v^6) \, dx dy
+ 5 \nu_1 \int_{\Omega} (u^4 |u_x|^2 + v^4 |v_x|^2) \, dx dy
+ 5 \nu_2 \int_{\Omega} (u^4 |u_y|^2 + v^4 |v_y|^2) \, dx dy \notag \\
& = -\int_{\Omega} p_x \, u^5 \, dx dy
- \int_{\Omega} p_y \, v^5 \, dx dy
- \int_{\Omega} q_x \, u^5 \, dx dy
- \int_{\Omega} q_y \, v^5 \, dx dy.
\end{align*}

The terms on the right-hand side can be estimated as follows. Applying Proposition \ref{triple}, we obtain:
\begin{align*}
\left| \int_{\Omega} p_x \, u^5 \, dx dy \right|
= \left| 5 \int_{\Omega} p \, u_x \, u^4 \, dx dy \right| \leq C \|p\|_2^{1/2} \|p_y\|_2^{1/2} \, \|u^2\|_2^{1/2} \|u u_x\|_2^{1/2} \left\|u^2 u_x \right\|_2.
\end{align*}

Similarly,
\begin{align*}
\left| \int_{\Omega} p_y \, v^5 \, dx dy \right|
\leq C \|p_y\|_2 \, \|v^2\|_2^{1/2} \|v v_y\|_2^{1/2} \, \|v^3\|_2^{1/2} \left\|v^2 v_x \right\|_2^{1/2}.
\end{align*}

Moreover, applying Proposition \ref{triple} and using (\ref{grad-q}), we obtain:
\begin{align*}
\left| \int_{\Omega} q_x \, u^5 \, dx dy \right|
&\le C \|\nabla q\|_2 \, \|u^2\|_2^{1/2} \|u u_y\|_2^{1/2} \, \|u^3\|_2^{1/2} \left\|u^2 u_x \right\|_2^{1/2} \notag \\
&\le C \nu_2 \|u_{yy}\|_2^{1/2} \|u_{yx}\|_2^{1/2} \, \|u^2\|_2^{1/2} \|u u_y\|_2^{1/2} \, \|u^3\|_2^{1/2} \left\|u^2 u_x \right\|_2^{1/2} \notag \\
&\le C \left( \nu_2 \|u_{yx}\|_2^2 \right)^{1/4} \left( \nu_2^2 \|u_{yy}\|_2^2 \right)^{1/4} \, \|u^2\|_2^{1/2} \left( \nu_2 \|u u_y\|_2^2 \right)^{1/4} \, \|u^3\|_2^{1/2} \left\|u^2 u_x \right\|_2^{1/2}.
\end{align*}

Similarly,
\begin{align*}
\left| \int_{\Omega} q_y \, v^5 \, dx dy \right|
&\le C \|\nabla q\|_2 \, \|v^2\|_2^{1/2} \|v u_x\|_2^{1/2} \, \|v^3\|_2^{1/2} \|v^2 v_x\|_2^{1/2} \notag \\
&\le C \nu_2 \|u_{yy}\|_2^{1/2} \|u_{yx}\|_2^{1/2} \, \|v^2\|_2^{1/2} \|v u_x\|_2^{1/2} \, \|v^3\|_2^{1/2} \|v^2 v_x\|_2^{1/2} \notag \\
&\le C \left( \nu_2 \|u_{yx}\|_2^2 \right)^{1/4} \left( \nu_2^2 \|u_{yy}\|_2^2 \right)^{1/4} \, \|v^2\|_2^{1/2} \left( \nu_2 \|v u_x\|_2^2 \right)^{1/4} \, \|v^3\|_2^{1/2} \|v^2 v_x\|_2^{1/2}.
\end{align*}

Thus, combining the above estimates and applying Young's inequality, we obtain:
\begin{align*}
& \frac{d}{dt} \left( \|u\|_6^6 + \|v\|_6^6 \right)
+ \nu_1 \left( \|u^2 u_x\|_2^2 + \|v^2 v_x\|_2^2 \right)
+ \nu_2 \left( \|u^2 u_y\|_2^2 + \|v^2 v_y\|_2^2 \right) \\
& \leq C_{\nu_1} \big[\|p\|_2^2 \|p_y\|_2^2 +  \|u^2\|_2^2 \|u u_x\|_2^2
+  \|p_y\|_2^2 +  \|v^2\|_2^2 \|v u_x\|_2^2 \|v^3\|_2^2  +  \nu_2 \|u_{yx}\|_2^2 +  \nu_2^2 \|u_{yy}\|_2^2 \\
& \quad \quad \quad +  (\|u^2\|_2^2 + \|v^2\|_2^2) \left( \nu_2 \|u u_y\|_2^2 + \nu_2 \|v u_x\|_2^2 \right)
(\|u^3\|_2^2 + \|v^3\|_2^2) \big].
\end{align*}

Thanks to Gr\"onwall's inequality and the bounds (\ref{u2-7}) and (\ref{ebase}), we obtain:
\begin{align*}
\|u\|_6^6 + \|v\|_6^6
+ \nu_1 \int_0^t \left( \|u^2 u_x\|_2^2 + \|v^2 v_x\|_2^2 \right) \, ds
+ \nu_2 \int_0^t \left( \|u^2 u_y\|_2^2 + \|v^2 v_y\|_2^2 \right) \, ds
\leq M_1(\nu_1).
\end{align*}

As a result, using (\ref{PP}), we conclude:
\begin{align}     \label{p3est}
\|p\|_3^3 \leq C \left( \|u\|_6^6 + \|v\|_6^6 \right)
\leq M_1(\nu_1).
\end{align}

\vspace{0.1 in}

\subsection{$\| u\|_{2r}$ Estimates for $r > 1$}

We begin by decomposing the pressure term $p$ into low- and high-frequency components. Specifically, for any $R > 0$, we write:
\[
p = \bar{p} + \tilde{p},
\]
where the decomposition satisfies the following estimates:
\begin{align}
& \|\bar{p}\|_{\infty}^2 \leq \log R \, \| \nabla p \|_2^2,  \label{urr-1}\\
& \|\tilde{p}\|_2 \leq \frac{C}{R} \, \| \nabla p \|_2. \label{urr-2}
\end{align}

As a consequence, we obtain:
\begin{align} \label{tp3}
\|\tilde{p}\|_3 \leq C\|\tilde{p}\|_2^{\frac{2}{3}} \, \| \nabla p \|_2^{\frac{1}{3}}
\leq C R^{-\frac{2}{3}} \, \| \nabla p \|_2.
\end{align}

Taking the inner product of the first equation in (\ref{vbou}) with $u^{2r-1}$ and integrating by parts, we obtain:
\begin{align}  \label{urr-3}
&\frac{d}{dt} \|u^r\|_2^2
+ \nu_1 \| (u^r)_x \|_2^2
+ \nu_2 \| (u^r)_y \|_2^2 \notag \\
&\leq C r \left| \int_{\Omega} p_x \, u^{2r-1} \, dx dy + \int_{\Omega} q_x \, u^{2r-1} \, dx dy \right|  \notag\\
&\leq C r \Big[ (2r - 1) \int_{\Omega} |\bar{p} \, u_x \, u^{2(r-1)}| \, dx dy
+ (2r - 1) \int_{\Omega} |\tilde{p} \, u_x \, u^{2(r-1)}| \, dx dy \notag\\
&\hspace{0.4 in} + \int_{\Omega} |q_x \, u^{2r-1}| \, dx dy \Big].
\end{align}

The terms on the right-hand side are estimated as follows:
\begin{align} \label{urr-4}
\int_{\Omega} |\bar{p} \, u_x \, u^{2r-2}| \, dx dy
= \frac{1}{r} \int_{\Omega} |\bar{p} \, (u^r)_x \, u^{r-1}| \, dx dy \leq \frac{1}{r} \|\bar{p}\|_{\infty} \, \|u^{r-1}\|_2 \, \| (u^r)_x \|_2.
\end{align}

Using Proposition \ref{triple}, we estimate:
\begin{align}  \label{urr-5}
\int_{\Omega} |\tilde{p} \, u_x \, u^{2r-2}| \, dx dy
&= \frac{1}{r} \int_{\Omega} |\tilde{p} \, (u^r)_x \, u^{r-1}| \, dx dy  \notag\\
&\leq \frac{C}{r} \|\tilde{p}\|_3^{3/5} \, \|\tilde{p}_y\|_2^{2/5} \, \|u^{r-1}\|_2^{3/5} \, \|(u^{r-1})_x\|_2^{2/5} \, \|(u^r)_x\|_2 \notag \\
&\leq \frac{C}{r} \|\tilde{p}\|_3^{3/5} \, \|\tilde{p}_y\|_2^{2/5} \, \|u^{r-1}\|_2^{3/5} \, \|u_x\|_2^{\frac{2}{5(r-1)}} \, \|(u^r)_x\|_2^{\frac{2(r-2)}{5(r-1)}} \, \|(u^r)_x\|_2,
\end{align}
where we used Hölder's inequality in the last step.

Also, using the interpolation inequality, we deduce:
\begin{align} \label{urr-6}
\int_{\Omega} |q_x \, u^{2r-1}| \, dx dy
&\leq \|q_x\|_2 \, \|u^{2r-1}\|_2
= \|q_x\|_2 \, \|u^r\|_{\frac{2(2r-1)}{r}}^{\frac{2r-1}{r}}    \notag \\
&\leq C \|q_x\|_2 \, \|u^r\|_2 \, \|(u^r)_y\|_2^{\frac{r-1}{2r}} \, \|(u^r)_x\|_2^{\frac{r-1}{2r}}.
\end{align}

Combining the above estimates (\ref{urr-3})-(\ref{urr-6}) and applying Young's inequality, we obtain:
\begin{align} \label{urr-7}
&\frac{d}{dt} \|u^r\|_2^2
+ \nu_1 \|(u^r)_x\|_2^2
+ \nu_2 \|(u^r)_y\|_2^2  \notag \\
&\leq C(\nu_1) \Big[r^2 \|\bar{p}\|_{\infty}^2 \, \|u^{r-1}\|_2^2
+  r^{10/3} \|\tilde{p}\|_3^2 \, \|\tilde{p}_y\|_2^{4/3} \, \|u^{r-1}\|_2^2 + \|u_x\|_2^2  \notag\\
&\hspace{0.6 in} +  r^{\frac{2r}{r+1}} \nu_2^{-\frac{r-1}{2(r+1)}} \|q_x\|_2^{\frac{2r}{r+1}} \, \|u^r\|_2^{\frac{2r}{r+1}}  \Big] \notag\\
&\leq C(\nu_1) \Big[r^2 \log R \, \|\nabla p\|_2^2 \, \|u^{r-1}\|_2^2
+  r^{10/3} R^{-4/9}   \|\tilde{p}\|_3^{4/3}     \|\nabla p\|_2^2 \, \|u^{r-1}\|_2^2 + \|u_x\|_2^2 \notag \\
&\hspace{0.6 in} +  r^{\frac{2r}{r+1}} \nu_2^{-\frac{r-1}{2(r+1)}} \|q_x\|_2^{\frac{2r}{r+1}} \, \|u^r\|_2^{\frac{2r}{r+1}} \Big]  ,
\end{align}
where we have used (\ref{urr-1}) and (\ref{tp3}) which implies:
$\|\tilde{p}\|_3^{2/3} \leq C R^{-4/9} \, \|\nabla p\|_2^{2/3}$.

Let $r^{\frac{7}{3}} R^{-\frac{4}{9}} = 1$. We then deduce:
\begin{align} \label{urr-8}
&\frac{d}{dt} \|u^r\|_2^2
+ \nu_1 \| (u^r)_x \|_2^2
+ \nu_2 \| (u^r)_y \|_2^2  \notag\\
&\leq C(\nu_1) \Big[r^2 \log r \, (1 + \|\tilde{p}\|_3^{4/3}) \, \|\nabla p\|_2^2 \, \|u^{r-1}\|_2^2
+ \|u_x\|_2^2  +  r^{\frac{2r}{r+1}} \nu_2^{-\frac{r-1}{2(r+1)}} \|q_x\|_2^{\frac{2r}{r+1}} \, \|u^r\|_2^{\frac{2r}{r+1}} \Big] \notag\\
&\leq C(\nu_1) \Big[ r^2 \log r \, (1 + \|\tilde{p}\|_3^2) \, \|\nabla p\|_2^2 \, \|u\|_2^{\frac{2}{r-1}} \, \|u^r\|_2^{\frac{2(r-2)}{r-1}}
+ \|u_x\|_2^2 \notag\\
&\hspace{0.6 in} +  r^{\frac{2r}{r+1}} \nu_2^{-\frac{r-1}{2(r+1)}} \|q_x\|_2^{\frac{2r}{r+1}} \, \|u^r\|_2^{\frac{2r}{r+1}} \Big].
\end{align}

If we let $\eta(t) = \|u^r\|_2^{\frac{2}{r}}$, then $\|u^r\|_2^2 = [\eta(t)]^r$, and thus $\frac{d}{dt} \|u^r\|_2^2 = r \eta^{r-1} \eta_t$.
Then, from (\ref{urr-8}), we have
\begin{align*}
r \eta^{r-1} \eta_t \leq\;  C(\nu_1) \Big[r^2 \log r \, (1 + \|\tilde{p}\|_3^2) \, \|\nabla p\|_2^2 \, \|u\|_2^{\frac{2}{r-1}} \, \eta^{\frac{r(r-2)}{r-1}} + \|u_x\|_2^2 +  r^{\frac{2r}{r+1}} \nu_2^{-\frac{r-1}{2(r+1)}} \|q_x\|_2^{\frac{2r}{r+1}} \, \eta^{\frac{r^2}{r+1}} \Big].
\end{align*}
Dividing both sides by $r \eta^{r-1}$ yields
\begin{align}  \label{eta1}
\eta_t \leq\; & C(\nu_1) \Big[r \log r \, (1 + \|\tilde{p}\|_3^2) \, \|\nabla p\|_2^2 \, \|u\|_2^{\frac{2}{r-1}} \, \eta^{-\frac{1}{r-1}} + \|u_x\|_2^2 \,\eta^{-(r-1)} +  r^{\frac{r-1}{r+1}} \nu_2^{-\frac{r-1}{2(r+1)}} \|q_x\|_2^{\frac{2r}{r+1}} \, \eta^{\frac{1}{r+1}} \Big] \notag \\
\leq\; & C(\nu_1) \Big[ r \log r \, \|\nabla p\|_2^2 + \|u_x\|_2^2 +  r \nu_2^{-\frac{r-1}{2(r+1)}} \|q_x\|_2^{\frac{2r}{r+1}} \, \eta^{\frac{1}{r+1}} \Big],
\end{align}
where we have used (\ref{p3est}), and without loss of generality, we assume that $\eta(t) \geq 1$ for all $t \geq 0$.

Next, let $\psi = \eta^{\frac{1}{r+1}}$, so that $\eta = \psi^{r+1}$ and $\eta_t = (r+1) \psi^r \psi_t$. Then (\ref{eta1}) implies
\begin{align} \label{psi-1}
\psi_t \leq C(\nu_1)\Big[ \log r \, \|\nabla p\|_2^2 + \|u_x\|_2^2 +  \nu_2^{-\frac{r-1}{2(r+1)}} \|q_x\|_2^{\frac{2r}{r+1}}\Big].
\end{align}

Using H\"older's inequality and (\ref{ebase2}), we obtain
\begin{align} \label{eta2}
\nu_2^{-\frac{r-1}{2(r+1)}} \int_0^t \|q_x\|_2^{\frac{2r}{r+1}} \, ds
&\leq  \nu_2^{-\frac{r-1}{2(r+1)}} \left( \int_0^t \|q_x\|_2^2 \, ds \right)^{\frac{r}{r+1}} t^{\frac{1}{r+1}} \notag \\
&\leq C(\nu_1) \nu_2^{-\frac{r-1}{2(r+1)}} \nu_2^{\frac{r}{2(r+1)}}   t^{\frac{1}{r+1}}  = C(\nu_1) \nu_2^{\frac{1}{2(r+1)}}  t^{\frac{1}{r+1}}   \leq C(\nu_1,T)     ,
\end{align}
since $\nu_2 \ll \nu_1$, for $t\in [0,T]$.

Integrating (\ref{psi-1}) over $[0, t]$ and using (\ref{L2}), (\ref{ebase}), and (\ref{eta2}), we obtain
\begin{align} \label{psi-2}
\psi(t) \leq \psi(0) + C(\nu_1,T) \log r       , \quad \text{for all } t \in [0,T].
\end{align}

Since $\eta = \psi^{r+1}$, it follows that
\begin{align} \label{eta-22}
\eta(t) \leq C(\nu_1,T) (\log r)^{r+1} , \quad \text{for all } t \in [0,T].
\end{align}

To obtain a sharper bound for $\eta$ in terms of $r$, we integrate (\ref{eta1}) over $[0, t]$:
\begin{align} \label{eta3}
\eta(t) \leq \eta(0) + C(\nu_1) r \log r + C(\nu_1) r \nu_2^{-\frac{r-1}{2(r+1)}} \int_0^t \|q_x\|_2^{\frac{2r}{r+1}} \eta^{\frac{1}{r+1}} \, ds.
\end{align}

Substituting (\ref{eta-22}) into (\ref{eta3}), we obtain
\begin{align*}
\eta(t) &\leq \eta(0) + C(\nu_1) r \log r + C(\nu_1, T) r \log r \, \nu_2^{-\frac{r-1}{2(r+1)}} \int_0^t \|q_x\|_2^{\frac{2r}{r+1}} \, ds \\
&\leq \eta(0) + C(\nu_1,T) r \log r,
\end{align*}
due to (\ref{eta2}). Thus,
\begin{align} \label{LR}
\|u\|_{2r}^2 = \eta(t) \leq C(\nu_1,T) r \log r,   \quad  \text{for}\; t\in [0,T].
\end{align}

Furthermore, using (\ref{LR}) and the standard estimate to the 2D NSE along with the Brezis-Wainger inequality
\begin{align} \label{BW}
\|f\|_{\infty}^2 \leq C \sup_{r} \frac{\|f\|_r^2}{r\log r}
\log (1+ \|f\|_{H^2}^2)\log(\log (e+ \|f\|_{H^2}^2)),
\end{align}
we obtain
\begin{align}
\|u\|_{\infty}^2  \label{L_inf}
\leq C(\nu_1,T) (e+ |\log \nu_2| ) \log (e+ |\log \nu_2| ).
\end{align}

\vspace{0.1 in}

\subsection{ $\| v\|_{2r}$ Estimates for $r>1$}

We prove this by induction. Suppose that
\begin{align}\label{INDUCT}
&  \max_{0\leq t \leq T} (\|u^r\|_2^2+\|v^r\|_2^2)  + \, \nu_1  \int_0^T \int_{\Omega}  (u^2+v^2)^{r-1} (u_x^2+v_x^{2})  dxdy dt
 \notag\\
&
+ \, \nu_2  \int_0^T \int_{\Omega} (u^2+v^2)^{r-1} (u_y^2+v_y^{2}) dx dy dt \leq N_r(T, r, \nu_1).
\end{align}
Note that (\ref{INDUCT}) holds for \( r = 1 \) due to (\ref{L2}).

Taking the inner product of the first equation in (\ref{vbou}) with $u(u^2+ v^2)^r$ and
the second with $v(u^2+ v^2)^r$, and integrating by parts, we obtain
\begin{align} \label{fu1}
& \frac{1}{2(r+1)} \frac{d}{dt} \int_{\Omega} (u^2+ v^2)^{r+1} \, dx dy \, + \, \nu_1  \int_{\Omega} (|u_x|^2+  |v_x|^{2}) (u^2+ v^2)^r \, dx dy   \notag\\
&\quad + \, \nu_2  \int_{\Omega} (|u_y|^2+ |v_y|^{2}) (u^2+ v^2)^r \, dx dy   \notag \\
& \leq \int_{\Omega} \left|(p_x +q_x) u (u^2+ v^2)^r + (p_y +  q_y  ) v  (u^2+ v^2)^r  \right| \, dx dy  \notag\\
& \leq  C 2^r \int_{\Omega} (|\nabla p| + |\nabla q|) \, (|u|^{2r+1} +|v|^{2r+1})  \, dx dy .
\end{align}
The terms on the right-hand side of (\ref{fu1}) can be estimated as follows. First, using the Cauchy-Schwarz inequality, we have
\begin{align} \label{fu2}
 \int_{\Omega} (|\nabla p| + |\nabla q|) \, |u|^{2r+1}   dx dy
 \le C ( \|\nabla p\|_2+ \|\nabla q\|_2) \, \|u\|_{2(2r+1)}^{2r+1}.
\end{align}
Applying Proposition \ref{triple}, we also have
\begin{align} \label{fu3}
&\int_{\Omega} (|\nabla p| + |\nabla q|) \, |v|^{2r+1} \,dx dy  =   \int_{\Omega} (|\nabla p| + |\nabla q|) \, |v|^{r+1} \, |v|^r  \,dx dy   \notag\\
&\le C ( \|\nabla p\|_2+ \|\nabla q\|_2) \, \|v^{r+1}\|_2^{1/2} \|(v^{r+1})_x\|_2^{1/2}
 \, \|v^{r}\|_2^{1/2} \|(v^{r})_y\|_2^{1/2}  \notag \\
& \le C r^{1/2} ( \|\nabla p\|_2+ \|\nabla q\|_2) \, \|v^{r+1}\|_2^{1/2} \|(v^{r+1})_x\|_2^{1/2}
 \, \|v^{r}\|_2^{1/2} \|(v^{r-1}) u_x\|_2^{1/2}.
\end{align}

Combining (\ref{fu1})-(\ref{fu3}), we obtain
\begin{align} \label{fu4}
& \frac{d}{dt} \int_{\Omega} (u^2+ v^2)^{r+1} dx dy \, + \, 2\nu_1 r  \int_{\Omega} (|u_x|^2+  |v_x|^{2}) (u^2+ v^2)^r dx dy   \notag\\
&  \quad + 2 \nu_2 r \int_{\Omega} (|u_y|^2+ |v_y|^{2}) (u^2+ v^2)^r dx dy   \notag \\
& \leq  C 2^r  ( \|\nabla p\|_2+ \|\nabla q\|_2) \, \|u\|_{2(2r+1)}^{2r+1}     \notag\\
& \quad +C 2^r ( \|\nabla p\|_2+ \|\nabla q\|_2) \, \|v^{r+1}\|_2^{1/2} \|(v^{r+1})_x\|_2^{1/2}
 \, \|v^{r}\|_2^{1/2} \|(v^{r-1}) u_x\|_2^{1/2}.
 \end{align}
Notice that
$ \|(v^{r+1})_x\|_2^2 = (r+1)^2 \int_{\Omega} |v_x|^2 v^{2r} dx dy$.
Then, applying Young's inequality to (\ref{fu4}), we obtain
\begin{align} \label{fu5}
& \frac{d}{dt} \int_{\Omega} (u^2+ v^2)^{r+1} dx dy \, + \,  \nu_1 r  \int_{\Omega} (|u_x|^2+  |v_x|^{2}) (u^2+ v^2)^r dx dy   \notag\\
& \quad   +\nu_2 r \int_{\Omega} (|u_y|^2+ |v_y|^{2}) (u^2+ v^2)^r dx dy   \notag \\
& \leq  C 2^r  ( \|\nabla p\|_2+ \|\nabla q\|_2) \, \|u\|_{2(2r+1)}^{2r+1}     \notag\\
& \quad +C 4^r \nu_1^{-1/3} ( \|\nabla p\|_2+ \|\nabla q\|_2)^{4/3} \, \|v^{r+1}\|_2^{2/3}
 \,\|v^{r}\|_2^{2/3} \|(v^{r-1}) u_x\|_2^{2/3} \notag\\
& \leq C(r,\nu_1) \left[ \|\nabla p\|_2^2 + \|\nabla q\|_2^2 +\|u\|_{2(2r+1)}^{2(2r+1)} +\|v^{r+1}\|_2^{2} \|v^{r}\|_2^{2} \|(v^{r-1}) u_x\|_2^{2}\right].
 \end{align}

Thanks to Grönwall’s inequality, the bounds (\ref{ebase})-(\ref{ebase2}) and (\ref{LR}), and the induction hypothesis (\ref{INDUCT}), we conclude that
\begin{eqnarray*}
&&  \max_{0\leq t \leq T} \int_{\Omega}  (u^2+v^2)^{r+1} dxdy
 + \, \nu_1  \int_0^T \int_{\Omega}  (u^2+v^2)^{r} (u_x^2+v_x^{2})  dxdy dt
 \\
&&
+ \, \nu_2  \int_0^T \int_{\Omega} (u^2+v^2)^{r} (u_y^2+v_y^{2}) dx dy dt
\leq   N_{r+1}(T, r, \nu_1).
\end{eqnarray*}
This completes the induction procedure, and thus (\ref{INDUCT}) holds for all $r\geq 1$.

\vspace{0.2 in}

\section{Vanishing vertical viscosity} \label{sec-vanish}
This section is devoted to proving Theorem \ref{main-thm} by using the energy estimates provided in Section \ref{sec-est}.
Let $(u,v)$ be the unique global solution of system (\ref{vbou}) with initial data $u_0,v_0 \in H^2(\Omega)$.
Let $(U,V)$ be the unique global solution of system (\ref{hbou}) with initial data $U_0,V_0 \in H^2(\Omega)$. Assume that $\|U\|_{H^2} + \|V\|_{H^2} \leq M$ on $[0,T]$, where the uniform bound $M$ is independent of $T$, as shown in our recent paper \cite{Cao} and in the work of Dong et al. \cite{DWXZ}. 

Let $\tilde u = u - U$ and $\tilde v = v - V$. By subtracting (\ref{vbou}) with (\ref{hbou}), we obtain
\begin{equation}\label{diff}
 \left\{
\begin{array}{l}
 \tilde u_t + \tilde u U_x  + u \tilde u_x  +    \tilde v U_y +  v \tilde u_y          + p_x +q_x - P_x = \nu_1 \tilde u_{xx} + \nu_2 \tilde u_{yy} + \nu_2 U_{yy}, \\
 \tilde v_t+ \tilde u V_x  +  u \tilde v_x  +    \tilde v V_y    +    v \tilde v_y  +      p_y  + q_y  - P_y       =  \nu_1 \tilde v_{xx} + \nu_2 \tilde v_{yy} +  \nu_2 V_{yy} , \\
 u_x+v_y=0,    \;\;\;      U_x+V_y=0,   \;\;\;      \tilde u_x + \tilde v_y=0,        \\
 u(x,0,t)=u_y(x,1, t)= v(x,0,t)=v(x,1,t)=0,\\
 V(x,0,t)=V(x,1,t)=0.
 \end{array} \right.
 \end{equation}
Therefore, we have
\begin{align}
&\tilde u(x,0,t) = - U(x,0,t), \;\;\;  \tilde u(x,1,t) = u(x,1,t) - U(x,1,t), \label{dif-1}\\
&\tilde u_y(x,0,t) = u_y(x,0,t) -  U_y(x,0,t),  \;\;\;  \tilde u_y(x,1,t) = -  U_y(x,1,t), \label{dif-2}\\
&\tilde v(x,0,t) =  \tilde v(x,1,t) = 0. \label{dif-3}
\end{align}

Taking the inner product of (\ref{diff}) with $\tilde u$ and $\tilde v$ yields
\begin{align}    \label{diff-1}
&\frac{1}{2} \frac{d}{dt} (\|\tilde u\|_2^2 + \|\tilde v\|_2^2) + \nu_1 (\|\tilde u_x\|_2^2 +  \|\tilde v_x\|_2^2) + \nu_2 (\|\tilde u_y\|_2^2 +  \|\tilde v_y\|_2^2) \notag\\
&= \nu_2 \int_{-\pi}^{\pi} (\tilde u_y \tilde u)|_{y=0}^{y=1} \, dx + \nu_2 \int_{\Omega} (U_{yy} \tilde u + V_{yy} \tilde v) \, dx dy  \notag\\
&\hspace{0.2 in}- \int_{\Omega} [\tilde u^2 U_x + \tilde v^2 V_y + (U_y + V_x) \tilde u \tilde v ] \, dx dy.
 \end{align}

Since we assume $\|U\|_{H^2(\Omega)}$, $\|V\|_{H^2(\Omega)} \leq M$ on $[0,T]$, then
\begin{align}    \label{diff-2}
\nu_2 \int_{\Omega} (U_{yy} \tilde u + V_{yy} \tilde v) \, dx dy  \leq  \nu_2 (\|U_{yy}\|_2 \|\tilde u\|_2 +  \|V_{yy}\|_2 \|\tilde v\|_2) \leq \nu_2 M  (\|\tilde u\|_2 + \|\tilde v\|_2).
 \end{align}

We calculate
\begin{align} \label{diff-3}
&\left|\int_{\Omega} [\tilde u^2 U_x + \tilde v^2 V_y + V_x  \tilde u \tilde v ] \, dx dy \right| = \left|\int_{\Omega} [\tilde u^2 U_x - \tilde v^2 U_x+ V_x  \tilde u \tilde v ] \, dx dy \right|  \notag\\
 & \leq 2 \int_{\Omega} \left(|\tilde u \tilde u_x U | + |\tilde v \tilde v_x U| + |V  \tilde u_x \tilde v| + |V  \tilde u \tilde v_x|   \right) \, dx dy  \notag\\
 &\leq   \frac{\nu_1}{4} (\|\tilde u_x\|_2^2 +  \|\tilde v_x\|_2^2)  + C \nu_1^{-1}(\|U\|_{\infty}^2 + \|V\|_{\infty}^2) (\|\tilde u\|_2^2 +  \|\tilde v\|_2^2 ) \notag\\
& \leq   \frac{\nu_1}{4} (\|\tilde u_x\|_2^2 +  \|\tilde v_x\|_2^2)  + C  \nu_1^{-1}   M^2 (\|\tilde u\|_2^2 +  \|\tilde v\|_2^2 ).
   \end{align}
Also, using Proposition \ref{triple}, we have
\begin{align}   \label{diff-4}
&\int_{\Omega} |U_y  \tilde u \tilde v| \, dx dy   \leq C \|U_y\|_2^{1/2} \|U_{xy}\|_2^{1/2} \|\tilde u\|_2  \|\tilde v\|_2^{1/2}    \|\tilde v_y\|_2^{1/2}
\leq C M  \|\tilde u\|_2  \|\tilde v\|_2^{1/2}    \|\tilde u_x\|_2^{1/2}   \notag\\
&\leq   \frac{\nu_1}{4} \|\tilde u_x\|_2^2 + C \nu_1^{-1/3} M^{4/3}  (\|\tilde u\|_2^2 +  \|\tilde v\|_2^2 ).
\end{align}

Using (\ref{diff-1})-(\ref{diff-4}) and (\ref{dif-1})-(\ref{dif-3}), we obtain
\begin{align}  \label{diff-44}
& \frac{d}{dt} (\|\tilde u\|_2^2 + \|\tilde v\|_2^2) + \nu_1 (\|\tilde u_x\|_2^2 +  \|\tilde v_x\|_2^2) + 2\nu_2 (\|\tilde u_y\|_2^2 +  \|\tilde v_y\|_2^2) \notag\\
&\leq \nu_2 \int_{-\pi}^{\pi} |U_y(1) \tilde u(1)| \, dx +  \nu_2 \int_{-\pi}^{\pi} |\tilde u_y(0) U(0)| \, dx \notag\\
& \quad + \nu_2 M  (\|\tilde u\|_2 + \|\tilde v\|_2) + C  \nu_1^{-1}     M^2  (\|\tilde u\|_2^2 + \|\tilde v\|_2^2).
\end{align}

Notice that
\begin{align*}
\int_{-\pi}^{\pi}  u(1)^2  \, dx =  \int_{-\pi}^{\pi}   \int_0^1 (u^2)_y  \,  dy     dx =   2 \int_{-\pi}^{\pi}   \int_0^1 u u_y  \, dy     dx \leq 2 \|u\|_2 \|u_y\|_2.
\end{align*}
Therefore, we have
\begin{align} \label{diff-5}
 &\int_{-\pi}^{\pi} |U_y(1) \tilde u(1)| \, dx  \leq  C\Big(\int_{-\pi}^{\pi} U_y(1)^2 \, dx  \Big)^{1/2}    \Big(\int_{-\pi}^{\pi}  u(1)^2  \, dx +    \int_{-\pi}^{\pi}  U(1)^2  dx  \Big)^{1/2}   \notag\\
 &  \leq C M  \left( \|u\|_2 \|u_y\|_2      +     M^2 \right)^{1/2} \leq C M  \|u_y\|_2^{1/2} + C M^2.
 \end{align}

Also,
\begin{align*}
\int_{-\pi}^{\pi}  u_y(0)^2  \, dx =  -\int_{-\pi}^{\pi}   \int_0^1 (u_y^2)_y   \, dy     dx = -  2 \int_{-\pi}^{\pi}   \int_0^1 u_y u_{yy} \,  dy     dx \leq 2 \|u_y\|_2 \|u_{yy}\|_2.
\end{align*}
Therefore,
\begin{align} \label{diff-6}
&\int_{-\pi}^{\pi} |\tilde u_y(0) U(0)| \, dx
\leq C\Big(\int_{-\pi}^{\pi}  u_y(0)^2  \, dx +    \int_{-\pi}^{\pi}  U_y(0)^2  dx  \Big)^{1/2}      \Big(\int_{-\pi}^{\pi} U(0)^2 dx  \Big)^{1/2}          \notag\\
&\leq  C( \|u_y\|_2 \|u_{yy}\|_2 + M^2)^{1/2} M  \leq C M  \|u_y\|_2^{1/2} \|u_{yy}\|_2^{1/2} + C M^2.
\end{align}

Combining the estimates (\ref{diff-44})-(\ref{diff-6}), we obtain that
\begin{align*}
&\frac{d}{dt} (\|\tilde u\|_2^2 + \|\tilde v\|_2^2) + \nu_1 (\|\tilde u_x\|_2^2 +  \|\tilde v_x\|_2^2) + 2\nu_2 (\|\tilde u_y\|_2^2 +  \|\tilde v_y\|_2^2) \notag\\
&\leq  C\nu_2 M^2 + C \nu_2 M  \|u_y\|_2^{1/2} +  C \nu_2 M  \|u_y\|_2^{1/2} \|u_{yy}\|_2^{1/2}  + C  \nu_1^{-1}     M^2  (\|\tilde u\|_2^2 + \|\tilde v\|_2^2).
\end{align*}

By Gr\"onwall's inequality, for all $t\in [0,T]$, we have
\begin{align} \label{diff-7}
&\|\tilde u(t)\|_2^2 + \|\tilde v(t)\|_2^2   \notag\\
&\leq e^{C  \nu_1^{-1}    M^2 t} \Big[ \|\tilde u_0\|_2^2 +  \|\tilde v_0\|_2^2 + C\nu_2 M^2 t  +   C \nu_2 M      \int_0^t   (\|u_y\|_2^{1/2} +  \|u_y\|_2^{1/2} \|u_{yy}\|_2^{1/2}) \, ds       \Big] \notag\\
&\leq  e^{C  \nu_1^{-1}      M^2 t} \Big[ \|\tilde u_0\|_2^2 +  \|\tilde v_0\|_2^2 + C\nu_2 M^2 t   \notag\\
& \hspace{0.2 in }+  C M   \nu_2^{3/4}   \int_0^t  \nu_2^{1/4}   \|u_y\|_2^{1/2} \, ds   + C M \nu_2^{1/4} \int_0^t  (\nu_2^{1/4}   \|u_y\|_2^{1/2}) (\nu_2^{1/2}    \|u_{yy}\|_2^{1/2}   ) \, ds  \Big] \notag\\
&\leq e^{C  \nu_1^{-1}      M^2 T}(\|\tilde u_0\|_2^2 +  \|\tilde v_0\|_2^2  + K \nu_2^{1/4}  )  \longrightarrow  0,
\end{align}
if $\nu_2 \rightarrow 0$ and $\|\tilde u_0\|_2^2 +  \|\tilde v_0\|_2^2 \rightarrow 0$, where we have used the bounds (\ref{L2}) and (\ref{u2-8}).

Furthermore, we obtain from (\ref{diff-7}), (\ref{LR}) and (\ref{INDUCT}) that
\begin{align*}
\|\tilde u\|_{2r} + \|\tilde v\|_{2r} &\leq    \left(\|\tilde u\|_{2} + \|\tilde v\|_{2} \right)^{\frac{1}{2r}}   \left(\|\tilde u\|_{4r-2} + \|\tilde v\|_{4r-2} \right)^{\frac{2r-1}{2r}}  \notag\\
&\leq  C(M, \nu_1, T) \left(\|\tilde u\|_{2} + \|\tilde v\|_{2} \right)^{\frac{1}{2r}}  \leq   C(M, \nu_1, T)  (\|\tilde u_0\|_2 +  \|\tilde v_0\|_2  + \nu_2^{1/8}  )^{\frac{1}{2r}}      \longrightarrow  0,
\end{align*}
if $\nu_2 \rightarrow 0$ and $\|\tilde u_0\|_2^2 +  \|\tilde v_0\|_2^2 \rightarrow 0$. This completes the proof of Theorem \ref{main-thm}.


\vspace{0.2 in}

\begin{thebibliography}{99}

\bibitem{Aydin}
M. C. Aydin. Inviscid limit on $L^p$-based Sobolev conormal spaces for the 3D Navier-Stokes equations with the Navier boundary conditions. arXiv:2502.03599


\bibitem{BNNT}
C. Bardos, T. T. Nguyen, T. T. Nguyen, and E. S. Titi.
The inviscid limit for the 2D Navier-Stokes equations in bounded domains.
\emph{Kinet. Relat. Models} 15 (2022), no. 3, 317–340.


\bibitem{BC}
H. Beir\~{a}o da Veiga and F. Crispo. Sharp inviscid limit results under Navier type boundary conditions. An $L^p$ theory. \emph{J. Math. Fluid Mech.} 12 (2010), no. 3, 397--411.







\bibitem{Cao}
C. Cao and Y. Guo. On the two-dimensional Navier-Stokes equations with horizontal viscosity. 2024.


\bibitem{Chemin-96}
J.-Y. Chemin. A remark on the inviscid limit for two-dimensional incompressible fluids. \emph{Comm. Partial Differential Equations} 21: 1771–1779, 1996.


\bibitem{CDGG} J.-Y. Chemin, B. Desjardins, Y. Gallagher, and E. Grenier. Fluids with anisotropic viscosity. \emph{Math. Model. Numer. Anal.} 34 (2000), 315--335.



\bibitem{C-Wu}
P. Constantin and J. Wu. Inviscid limit for vortex patches. \emph{Nonlinearity} 8: 735–742, 1995.


\bibitem{DWXZ} B. Dong, J. Wu, X. Xu, and N. Zhu. Stability and exponential decay for the 2D anisotropic Navier-Stokes equations with horizontal dissipation. \emph{J. Math. Fluid Mech.} 23 (2021), Paper No. 100, 11 pp.


\bibitem{Ebin}
D. Ebin and J. Marsden. Groups of diffeomorphisms and the notion of an incompressible fluid. \emph{Ann. Math.} 92: 102–163, 1970.


\bibitem{Fei}
M. Fei, T. Tao, and Z. Zhang.  On the zero-viscosity limit of the Navier-Stokes equations in $\mathbb R^3$ without analyticity. \emph{J. Math. Pures Appl.} (9) 112 (2018), 170–229.


\bibitem{Gol}
K. K. Golovkin. Vanishing viscosity in the cauchy problem for equations of hydrodynamics. \emph{Trudy Mat. Inst. Steklov.} 92: 31–49, 1966.


\bibitem{Iftimie}
D. Iftimie and G. Planas. Inviscid limits for the Navier-Stokes equations with Navier friction boundary conditions.
\emph{Nonlinearity} 19 (2006), no. 4, 899–918.

\bibitem{Kato-72}
T. Kato. Nonstationary flows of viscous and ideal fluids in $\mathbb R^3$. \emph{J. Functional Analysis} 9: 296–305, 1972.


\bibitem{Kato-84}
T. Kato. Remarks on zero viscosity limit for nonstationary Navier-Stokes flows with boundary. In \emph{Seminar on nonlinear partial differential equations} (Berkeley, Calif., 1983), volume 2 of Math. Sci. Res. Inst. Publ., pages 85–98. Springer, New York, 1984.


\bibitem{Liu}
C. Liu and Y. Wang. The inviscid limit of the incompressible anisotropic Navier-Stokes equations with the non-slip boundary condition.
\emph{Z. Angew. Math. Phys.} 64 (2013), no. 4, 1187–1225.


\bibitem{Mae}
Y. Maekawa. On the inviscid limit problem of the vorticity equations for viscous incompressible flows in the half-plane. \emph{Comm. Pure Appl. Math.} 67(7): 1045–1128, 2014.






\bibitem{Mas}
N. Masmoudi. The Euler limit of the Navier-Stokes equations, and rotating fluids with
boundary. \emph{Arch. Rational Mech. Anal.} 142 (1998), 375–394.






\bibitem{Prandtl}
D. Prandtl. Uber Flussigkeitsbewegung bei sehr kleiner Reibung.  \emph{Verh. III Intern. Math. Kongr. Heidelberg}, pages 485–491, 1904.



\bibitem{Caf}
M. Sammartino and R. E. Caflisch. Zero viscosity limit for analytic solutions of the Navier-Stokes equation on a half-space. I, II. 
\emph{Comm. Math. Phys.} 192 (1998), 433–491.

\bibitem{Swann}
H. Swann. The convergence with vanishing viscosity of non-stationary Navier-
Stokes flow to ideal flow in $\mathbb R^3$. \emph{Trans. Amer. Math. Soc.} 157 (1971), 373–397.

\bibitem{Tao}
T. Tao. Vanishing vertical viscosity limit of anisotropic Navier-Stokes equation with no-slip boundary condition.
\emph{J. Differential Equations} 265 (2018), no. 9, 4283–4310.



\end{thebibliography}
\end{document}